\documentclass[a4paper,11pt]{article}
\usepackage[utf8]{inputenc}

\usepackage{boxedminipage}
\usepackage{amsfonts}
\usepackage{amsmath} 
\usepackage{amssymb}
\usepackage{graphicx}
\usepackage{amsthm}
\usepackage{t1enc}
\usepackage{subfig}

\newtheorem{theorem}{Theorem}[section]
\newtheorem{lemma}[theorem]{Lemma}

\newtheorem{definition}[theorem]{Definition}

\newcommand{\forceP}{\mathbb{P}}
\newcommand{\forceQ}{\mathbb{Q}}
\newcommand{\forceR}{\mathbb{R}}

\newcommand{\ZFC}{\mathsf{ZFC}}

\newcommand{\ZFP}{\mathsf{ZF}^-}

\newcommand{\PFA}{\mathsf{PFA}}
\newcommand{\BPFA}{\mathsf{BPFA}}

\newcommand{\CH}{\mathsf{CH}}

\newcommand{\PD}{\mathsf{PD}}

\newcommand{\Lim}{\mathop{\textrm{Lim}}}

\newcommand{\Cof}{\textrm{Cof}}

\newcommand{\mbp}{\mathbb{P}}

\def\undertilde#1{\mathord{\vtop{\ialign{##\crcr
$\hfil\displaystyle{#1}\hfil$\crcr\noalign{\kern1.5pt\nointerlineskip}
$\hfil\tilde{}\hfil$\crcr\noalign{\kern1.5pt}}}}}

\title{The global $\Sigma^1_{n+1}$-Uniformization Property and $\BPFA$} 
\author{ Stefan Hoffelner\footnote{Supported by the Deutsche Forschungsgemeinschaft (DFG German Research Foundation) under Germanys Excellence Strategy EXC 2044 390685587, Mathematics M\"unster: Dynamics-Geometry-Structure.}  }

\begin{document}

\maketitle

\begin{abstract}
We show  that given a reflecting cardinal, one can produce a model of $\mathsf{BPFA}$ where the $\Sigma^1_n$-uniformization property holds simultaneously for all $n \ge 2$.
\end{abstract}

\section{Introduction}

Given $A \subset \omega^{\omega} \times \omega^{\omega}$, we say that $f$ is a uniformization (or a uniformizing function) of
$A$ if there is a function $f$ such that 
\[dom(f)= pr_1(A)= \{ x \in 2^{\omega} \, : \, \exists y ((x,y) \in A \} \] and the graph of $f$ is a subset of $A$.
\begin{definition}[Uniformization Property]
 We say that the projective pointclass $\Gamma \in \{\Sigma^1_n \mid n \in \omega\} \cup \{\Pi^1_n \mid n \in \omega\}$ has the uniformization
 property iff every $\Gamma$-set in the plane admits a uniformization 
 whose graph is in  $\Gamma$, i.e. the relation $(x,y) \in f$ is in $\Gamma$.
 \end{definition}

It is a classical result due to M. Kondo that lightface $\Pi^1_{1}$-sets do have the 
uniformization property, this also yields the uniformization property for $\Sigma^1_{2}$-sets. This is all $\ZFC$ can prove about the uniformization property of projective sets. In the constructible universe $L$, for every $n \ge 3$, $\Sigma^1_{n}$ does have the uniformization property which follows from the existence of a good wellorder by an old result of Addison (see \cite{Addison}). Recall that a $\Delta^1_n$-definable wellorder $<$ of the reals is a \emph{good} $\Delta^1_n$-wellorder if $<$ is of ordertype $\omega_1$ and the relation $<_{I} \subset (\omega^{\omega})^2$ defined via
\[ x <_I y \Leftrightarrow \{(x)_n \, : \, n \in \omega\}= \{z \, : \, z < y\} \]
where $(x)_n$ is some fixed recursive partition of $x$ into $\omega$-many reals, is a $\Delta^1_n$-definable. It is easy to check that the canonical wellorder of the reals in $L$ is a good $\Delta^1_2$-wellorder so the $\Sigma^1_n$-uniformization property follows for $n \ge 2$. The very same argument also shows that inner models of the form $L[U]$ for $U$ a normal, $\kappa$-complete ultrafilter on $\kappa$ share the same $L$-like pattern for uniformization.

On the other hand, large cardinal assumptions draw a very different picture.  Due to the celebrated result of Moschovakis (see \cite{Kechris} 39.9), $\PD$ implies that $\Pi^1_{2n+1}$ and $\Sigma^1_{2n+2}$-sets have the uniformization property for $n >1$. By the famous Martin-Steel result (see \cite{MS}, Theorem 13.6.), the assumption of infinitely many Woodin cardinals outright implies $\PD$, and hence large cardinals fully settle the behaviour of the uniformization property within the projective hierarchy.

The connection of $\PD$ with forcing axioms is established via core model induction. Under the assumption of the proper forcing axiom, Schimmerling, Steelt and Woodin showed  that $\PD$ is true (in fact much more is true, see \cite{Steel}), thus also under $\PFA$ the $\Pi^1_{2n+1}$-uniformization holds for $n>1$. 
As the uniformization property for one pointclass rules out the uniformization property of the dual pointclass, the behaviour of sets of reals in $L$ and under $\PFA$ contradict each other.

It is natural to investigate uniformization in the presence of weaker forcing axioms, the bounded proper forcing axiom $\BPFA$ (introduced by Goldstern-Shelah, see \cite{GS})  being the paradigmatic example together with Martin's Axiom. 
A first step was taken in \cite{Ho1}, where it is shown that $\BPFA$ and ``$\omega_1$ is not inaccessible to reals$"$ outright imply that the boldface $\bf{\Sigma}^1_3$-uniformization property holds.
However
Addisons result from 1959 has been the only tool to obtain the $L$-like global uniformization pattern. As a consequence the known universes which satisfy the $\Sigma^1_n$ uniformization property for all $n>1$ are quite special, in particular they must satisfy $\CH$, hence $\BPFA$ fails in such universes.

The main goal of this paper is to produce a new method which will force the $L$-like uniformization pattern. We will introduce a family of $\Sigma^1_n$-predicates whose truth value can be switched from false to true using carefully designed forcings which are $\omega_1$-preserving but not proper. These predicates in turn will be used to define an iterated forcing construction yielding a universe in which the $\Sigma^1_n$-uniformization property holds for each $ n \ge 2$ simultaneously. The definitions of  the uniformizing functions are robust under additional usage of proper forcings, which opens the possibility of forcing $\BPFA$ simultaneously. 
\begin{theorem}
Working in $L$, assuming the existence of a reflecting cardinal, there is a generic extension of $L$ in which $\BPFA$ and the $\Sigma^1_n$-uniformization property for every $n \ge 2$ are true.
\end{theorem}

The forcing method which yields the $\Sigma^1_n$-uniformization property seems to be quite flexible and we expect several further applications. 

We end the introduction with a short description of how the article is organized. Section 1 introduces the forcings we will use. This builds on a technique first introduced in \cite{Ho5} which in turn relies on the notion of mutual stationarity by Foreman-Magidor and Todorcevic's argument that $\PFA$ implies the failure of square. Section 1.3 sets up a coding machinery which can be combined with the arguments for obtaining $\BPFA$ without dangerous interferences. Section 2 introduces a family of $\Sigma^1_n$-formulas which we will use for the uniformization property. Section 3 explains the coding method which produces universes with the $L$-like uniformization property, and section 4 combines this method with the argument to obtain $\BPFA$ to finally prove the main theorem of this work.

It belongs to a series of articles which devote themselves to the study of the separation, the reduction and the uniformization property (see \cite{Ho2},\cite{Ho4}, \cite{Ho1} and \cite{Ho3}). We emphasize however that our arguments in this work (which deal with the $\Sigma$-side of uniformization) must differ substantially from the arguments which deal with the $\Pi$-side, (as \cite{Ho2},\cite{Ho4} and \cite{Ho3}) as the two pictures necessarily can not be combined within $\ZFC$.

\subsection{Mutually stationary sets}

An algebra $\mathcal{A}$ on a cardinal $\lambda$ is a structure of the form
$(\lambda, f_i)_{i\in \omega}$, where for each $i \in \omega$, $f_i: [\lambda]^{<\omega} \rightarrow \lambda$. Typically, the $f_i$'s will be definable Skolem functions over structures of the form $(H(\theta), \in, <)$ where $<$ is a well-order of $H(\theta)$ and $\theta$ is some regular cardinal. The following notion was introduced by M. Foreman and M. Magidor in their seminal \cite{FM}.

\begin{definition}
Let $K$ be a collection of regular cardinals whose supremum is strictly below $\kappa$, and suppose that $S_\eta\subseteq \eta$ for each $\eta \in K$. Then the collection of sets $\{ S_\eta \mid \eta\in K \}$ is \emph{mutually stationary} if and only if for all  algebras $\mathcal{A}$ on $\kappa$, there is an  $N\prec \mathcal{A}$ such that $$       \text{for all }\eta\in K\cap N,\  \sup(N\cap \eta)\in S_\eta.$$
\end{definition}

Foreman-Magidor (\cite{FM}) show that every sequence $\vec S= (S_{\eta} \mid \eta \in K)$ of stationary sets which concentrate on ordinals of countable cofinality is mutually stationary. For a fixed sequence $\vec S$ of stationary sets, let $\mathcal{T}_{\vec{S}}$ be the collection of all countable $N$ such that for all $\eta_i\in N$,  $\sup(N\cap \eta_i)\in S_i.$

\begin{theorem}[Foreman-Magidor]\label{thm:fm}
    Let $( \eta_i \mid i<j )$ be an increasing sequence of regular cardinals. Let $\vec S= ( S_i \mid    i<j  )$ be a sequence of stationary sets such that $S_i\subseteq \eta_i \cap \Cof(\omega)$. If $\theta$ is a regular cardinal greater than all $\eta_i$ and $\mathcal{A}$ is an algebra on $\theta$, then there is an $N\prec \mathcal{A}$ which  belongs to the class $\mathcal{T}_{\vec{S}}$. Hence $\vec S$ is mutually stationary.
\end{theorem}

From now on, we assume all stationary subsets of ordinals discussed in this section are concentrated on countable cofinality. The corresponding notion for being club in this context is that of an unbounded set which is closed under $\omega$-sequences.

\begin{definition}
    Suppose $\vec S=\{ S_\eta\mid \eta\in K\}$ is mutually stationary and that for every $\eta \in K$, $S_{\eta}$ is stationary, co-stationary in $\eta$ $ \cap $ Cof $ ({\omega})$. We say a forcing poset $\mbp$ is $\vec S$-preserving if the following holds:
    Suppose $\theta> 2^{ \vert \mbp^+\vert}$ is regular. Suppose $M$ is a countable elementary submodel of $H(\theta)$ with $\{\mbp,\vec S\}\subset M$ and $M\in \mathcal{T}_{\vec{S}}$. Suppose $p\in \mbp\cap M$. Then there exists a $(M,\mbp )$-generic condition $q$ extending $p$.
\end{definition}
We add some remarks concerning the notions:
    \begin{enumerate}
        \item Any proper forcing is $\vec S$-preserving.
        \item When $K=\{ \omega_1\}$ and $\vec{S}=S \subset \omega_1$, the definition of $\vec S$-preserving is identical to the usual definition of $S$-proper forcing.
        \item Let $\vec{S}$ be such that each $S_{\eta} \in \vec S$ is stationary, co-stationary in $\eta \cap Cof(\omega)$.
       Then an example of a non-proper, $\vec S$ preserving forcing is the forcing poset $Club(S_\eta)$ for a fixed $\eta$, i.e, the forcing which adds an unbounded subset to $S_\eta$ which is closed under $\omega$-sequences, via countable approximations.
    \end{enumerate}

The preservation theorems for countable support iterations of proper forcings can be generalized to $\vec S$-preserving forcings.

\begin{lemma}\label{lem: preservation}
    If $\langle P_i, \dot{Q}_i \mid  i< \alpha \rangle$ is a countable support iteration of  forcing notions and for each $i<\alpha$, $\Vdash_{P_i}$``$\dot{Q}_i$ is $\vec S$-preserving'' then $P_\alpha$ is $\vec S$-preserving.
\end{lemma}
\begin{proof}(Sketch, following the proof of \cite{Jech}, Theorem 31.15, in particular Lemma 31.17)
    We will only need to show by induction on $j\le \alpha$ that for any $N\in \mathcal T_{\vec{S}}$, if $j, \langle P_i, \dot{Q}_i \mid  i< \alpha \rangle\in N$, then:
\begin{itemize}
\item[$(\ast)_N$]  For every $q_0\in N\cap P_j$ that is $(N,P_j)$-generic and every $\dot{p} \in V^{P_j}$ such that
\[q_0 \Vdash_j \dot{p} \in  (P_{\alpha} \cap N) \land \dot{p} \upharpoonright j \in \dot{G}_j \]

there is an $(N,P_{\alpha})$ generic condition $q \in P_{\alpha}$  extending $q_0$, that is $q \upharpoonright j = q_0$ and  $q \Vdash_{\alpha} \dot{p} \in \dot{G}_{\alpha}$.
\end{itemize}

    The statement $(\ast)_N$ is identical to Lemma 31.17 in \cite{Jech}. It can be checked that the original proof also works here, which gives the iteration theorem exactly as in the proof of Theorem 31.15.
\end{proof}

We will use club shooting forcings relative to definable sequences $\vec S$ of mutually stationary sets to code information. These codings do not interfere with the proper forcings we will use to work towards $\BPFA$.
\begin{definition}
    Let $\kappa$ be an inaccessible cardinal. Let $\vec S= ( S_{i} \mid i< \kappa )$ be mutually stationary.  We say a forcing poset $\mbp$ is an $ \vec S$-coding if $\delta \le \kappa$ and $\mbp= \langle \forceP_\alpha, \dot{\forceQ}_\alpha \mid \alpha<\delta \rangle $ satisfies the following:
    \begin{itemize}
        \item $\mbp$ is a countable support iteration.
        \item For any $\alpha<\delta$, one of the followings holds:
        \begin{enumerate}

            \item Assume that $\alpha$ is inaccessible and $\mbp_\alpha$ is forcing equivalent to a forcing of size less than or equal to $\alpha$.\footnote{We say two forcing $P$ and $Q$ are equivalent if their Boolean completions $B(P)$ and $B(Q)$ are isomorphic.} Assume that in $V^{\mbp_\alpha}$, $( B_\beta \subset \alpha \mid \beta <2^\alpha )$ is an enumeration of an arbitrary set $X \subset P( {\alpha} )$. Then $\dot{\forceQ}_\alpha$ is allowed to be the countably supported product $\prod_{\beta < 2^{\alpha}} \dot{\forceR}_{\beta}$, where each $\dot{\forceR}_{\beta}$ is itself defined to be $$\prod_{j\in B_\beta}Club(S_{\alpha\cdot(\beta+1)+2j}) \times \prod_{j\notin B_\beta}Club(S_{\alpha\cdot(\beta+1)+2j+1})$$ using countable support.

             \item In all other cases, we have that $\Vdash_{\mbp_\alpha}\dot{\forceQ}_\alpha$ is proper.
        \end{enumerate}
    \end{itemize}
    Let $\eta$ be an regular cardinal, we  say $\mbp$ is an $\eta$-$\vec S$ coding if (1) is replaced by
    \begin{itemize}
        \item[(1')]  $\alpha \ge \eta$ is inaccessible and $\mbp_\alpha$ is forcing equivalent to a forcing of size less than or equal to $\alpha$. In $V^{\mbp_\alpha}$, $\langle B_\beta\mid \beta <2^\alpha\rangle$ is an enumeration of $P(\alpha)$. Then $\dot{\forceQ}_{\alpha}$ is allowed  to be the countably supported product $\prod_{\beta< 2^{\alpha}} \dot{\forceR}_{\beta}$ where each $\dot{\forceR}_{\beta}$ is itself $$\prod_{j\in B_\beta}Club(S_{\alpha\cdot(\beta+1)+2j}) \times \prod_{j\notin B_\beta}Club(S_{\alpha\cdot(\beta+1)+2j+1}). $$
    \end{itemize}
\end{definition}
By Lemma \ref{lem: preservation}, once we can show that every factor of an $\vec{S}$-coding is $\vec{S}$-preserving, we can deduce that  if $\mbp$ is a $\vec S$-coding forcing, then $\mbp$ is $\vec S$-preserving. This assertion follows from the proof of the next lemma which says that we will not accidentally code unwanted information whenever we use a $\vec{S}$-coding forcing.

\begin{lemma}\label{lem: stationary preservation}
    Suppose that $\vec S$ is stationary, co-stationary. Suppose $\mbp$ is an $\vec S$-coding forcing. Then for any $i \in \kappa$, the followings are equivalent:
    \begin{itemize}
        \item[(a)] $\Vdash_{\mbp} S_{i}$ contains an $\omega$-club.
        \item[(b)] there are $\beta,\alpha,j$ and sets $B_{\beta} \in X  \subset V_{\alpha}$ such that $j \in B_{\beta}$ if $\beta\cdot(\alpha+1)+2j=i$ and $i$ is even and $j \notin B_{\beta}$ if $\beta\cdot(\alpha+1)+2j+1=i$ and  $i$ is odd.
    \end{itemize}
\end{lemma}
\begin{proof} ((b) $\rightarrow$ (a)) Follow from the definition of the forcing.

    ((a) $\rightarrow$ (b)) Fix an $i$ and assume without loss of generality that $i$ is even. Write $i= \beta \cdot ({\alpha+1}) +2j$ and suppose for a contradiction that  $j$ is not an element of $B_{\beta}$. By the definition of  $\vec S$ coding forcing, we must have added a club through $S_{\beta (\alpha+1)+2j+1}$ instead.  Let $\vec T$ be the sequence $\langle T_k \mid k<\kappa \rangle$, where $T_k=S_{k}$ if $k\neq i$ and $T_i=\eta_i\setminus S_{i}$. It follows from Theorem \ref{thm:fm} again that $\vec T$ is mutually stationary. We will prove that $\mbp$ is $\vec T$-preserving to derive a contradiction. Indeed, we shall see that $\vec T$-preservation implies that $\eta_i \setminus S_i$ must remain stationary after forcing with $\forceP$, yet $\forceP \Vdash `` S_i$ contains an $\omega$-club$"$ which is impossible.

    To see that $\vec{T}$ preserving forcings preserve the stationarity of every $S_{\eta_i} \in \vec{T}$, we only need to note that for any name $\dot C$ of a subset of $\eta_i$ which is unbounded and $\omega$-closed, and any countable elementary substructure $N$ which contains $\dot C$  and for  which $\sup (N\cap \eta_i)\in S_{\eta_i}$, any $(N,\mbp)$-generic condition $q$ forces $\dot C\cap ( S_{\eta_i})\neq \emptyset$.

  Next we show by induction that each $\dot{Q}_{\beta}$ is forced to be $\vec T$-preserving. Work in $V[G_\beta]$. If $\dot Q_{\beta}/G_\beta$ is proper, then it is also $\vec T$-preserving. Otherwise, (1) holds. Now $\dot{Q}_\beta/G_{\beta}$ is a countable support product of club adding forcings. Fix any $N\in \mathcal T_{\vec{T}}$ which is a countable substructure of $H(\theta)^{V[G_\beta]}$. For any $p\in N\cap \dot Q_{\beta}$, we can construct a countable decreasing sequence of conditions $\langle p_n \mid n<\omega \rangle$ meeting all dense set in $N$. Define $q$ coordinatewise by setting $q(j)$ to be the closure of $\bigcup_{n<\omega} p_n(j)$ if $i\in N$ and trivial otherwise. Note that any non-trivial $q(j)$ is equal to $\bigcup_{n<\omega} p_n(j)\cup \{\sup(N\cap \eta_j)\}$, where $\eta_j=\sup (S_j)$ is a regular cardinal. As $N\in \mathcal T_{\vec{T}}$ we have $\sup(N\cap \eta_{j})\in S_j$, whenever $q(j)$ is non-trivial. Hence $q<p$ is a condition witnessing that each factor of the iteration  is $\vec T$-preserving, so the iteration  $\forceP$ is $\vec T$-preserving as well. But now $S_i$ must remain stationary after we forced with $\forceP$, which is a contradiction to $(a)$.
\end{proof}

The proof also shows that $\vec S$-coding preserves stationary subset of $\omega_1$ if $ \sup(S_0)>\omega_1$. As a Corollary of Lemma \ref{lem: stationary preservation} and the definition of $\vec S$-coding, in any generic extension by $\vec S$-coding and any even $i$, at most one of $S_{i}$ and $S_{i+1}$ contains a club.

The next lemma follows immediately from the definitions so we skip its proof.
\begin{lemma}\label{lem: composition}
    Suppose $\mbp= \langle P_\alpha, \dot{Q}_\alpha \mid \alpha<\delta \rangle $ is a countable support iteration. Suppose for any $\alpha>0$, $\dot{Q}_\alpha$ is forced to be $\eta_\alpha$-$\vec S$ coding of length $l(\alpha)$, where $\eta_ \alpha=\max \{\vert \mbp_\alpha\vert^+, \Sigma_{\beta<\alpha}l(\beta) \}$. Also let $\eta_0$ be regular. Then $\mbp$ is forcing equivalent to an $\eta_0$-$\vec S$ coding.
\end{lemma}

\subsection{Specializing threads through $\square$-sequences}

We will briefly present a celebrated theorem, due to S. Todorcevic, which we use to
find $\Sigma_1$-definitions of regular $L$- cardinals. The argument is again entirely due to Todorcevic (see \cite{To2}).

\begin{theorem}[Todorcevic, see \cite{To1}]
Let $\kappa > \omega_1$ be a regular cardinal and let $\Gamma \subset \kappa$ be a set of limit ordinals such that $\{ \delta < \kappa \mid \text{cf}  \, (\delta) = \omega_1 \} \subset \Gamma.$ Let $(C_{\alpha} \mid \alpha \in \Gamma \}$ be a sequence of  subsets of $\kappa$ such that
\begin{enumerate}
\item $\forall \alpha \in \Gamma$ $C_{\alpha} \subset \alpha$ and $C_{\alpha}$ is a club,
\item if $\beta$ is a limit point of $C_{\alpha}$, then $\beta \in \Gamma$ and
$C_{\beta} = C_{\alpha} \cap \beta$.
\item There is no club $C \subset \kappa$ such that if $\alpha$ is a limit point of
of $C$, then $\alpha \in \Gamma$ and $C_{\alpha} = C \cap \alpha$.
\end{enumerate}

Then there is a proper forcing $\mathbb{T}^{\kappa}$ such that $\mathbb{T}^{\kappa}$ forces that in the generic extension, there is a closed subset $D$ of $\kappa$ of ordertype $\omega_1$ and a function $f: D \rightarrow \omega$ such that for all $\alpha, \beta$, if $\alpha, \beta \in D$ and $\alpha \in \operatorname{lim} C_{\beta}$, then
$f(\alpha) \ne f(\beta)$. 

\end{theorem}

\begin{proof}[sketch of a proof]
For $\alpha, \beta \in \Gamma$ we let $\beta <_T \alpha$ if and only if 
$\beta$ is a limit point of $C_{\alpha}$. This induces a tree ordering on $\Gamma$.

Next we turn to some very useful definitions. For a regular, sufficiently large cardinal $\lambda$, an elementary chain $(M_{\alpha} \prec H(\lambda) \mid \alpha < \omega_1)$ is a sequence of elementary submodels of $H(\lambda)$ which additionally 
satisfies that $M_{\alpha} \subset M_{\beta}$ and $M_{\alpha} \in M_{\beta}$ whenever $\alpha < \beta$; and $M_{\alpha} = \bigcup_{\beta < \alpha} M_{\beta}$ if $\alpha$ is a limit ordinal.

We now define the desired forcing $\mathbb{T}^{\kappa}$, which is a prototype for the extremely useful class of forcing with side conditions, as follows:
A condition $p \in \mathbb{T}^{\kappa}$ is a pair $( (N_{\alpha} \mid \alpha \in E) , f )$ where
\begin{enumerate}
\item $E$ is a finite subset of $\omega_1$ and there is an elementary chain $(M_{\alpha} \mid \alpha < \omega_1)$ of countable, elementary substructures of $H(\kappa^+)$ such that $N_{\alpha} = M_{\alpha}$ for each $\alpha \in E$.
\item $f$ is a specializing function, i.e. if $\delta_{N_{\alpha}} := \operatorname{sup} (N_{\alpha} \cap \kappa$ then $f$ is a function from $\{ \delta_{N_{\alpha}} \mid \alpha \in E \}$ into $\omega$ such that $f(\gamma) \ne f(\delta) $ whenever $\gamma <_T \delta$.
\end{enumerate}

The key assertion is now that $\mathbb{T}^{\kappa}$ is proper provided we assume that there is no $\kappa$-chain in the tree $T$ on $\Gamma$.

Assuming that $\mathbb{T}^{\kappa}$ is proper, it is standard argument to verify that
$\mathbb{P}$ indeed adds an $\omega_1$-closed set $D$ and a specializing $f :D \rightarrow \omega$.
\end{proof}

From now on we assume that $0^{\#}$ does not exist.
Let $(C_{\alpha} \mid \alpha \in \text{Sing} \}$ be Jensen's global $\square$-sequence in $L$; here Sing denotes the set of limit ordinals $\alpha$ for which $cf(\alpha)^L < \alpha$. Recall that
\begin{enumerate}
\item $\forall \alpha \in \text{Sing}  (C_{\alpha} \subset \alpha)$ and $C_{\alpha}$ is club.
\item $\forall \alpha \in \text{Sing} (\text{o.t.} (C_{\alpha}) <\alpha)$.
\item $\forall \beta \in \text{Lim} (C_{\alpha} ) 
 (\beta \in \text{Sing} \land C_{\beta} = C_{\alpha} \cap \beta)$.
\end{enumerate}

We define a tree $T$ on Sing via setting
$\alpha <_T \beta$ iff $\alpha$ is a limit point of $C_{\beta}$. Note that for a cardinal $\lambda$, the tree for the corresponding $\square_{\lambda}$-sequence does not have a $\lambda^{+}$-branch, thus $\square_{\lambda}$ satisfies the three properties from Todorcevic's theorem.
 As a consequence, the proper poset $\mathbb{T}^{\lambda^{+}}$ adds a closed set $D$ in $(\lambda^{+})^L$ 
of ordertype $\omega_1$ and simultaneously specializes the tree $T$ restricted on ordinals in $D$,  i.e. the forcing adds a function $f: D \rightarrow \omega$ such that for $\alpha , \beta \in D$, if $\alpha <_T \beta$ then $f(\alpha) \ne f (\beta)$.

Consider now the $\Sigma_1$-formula $\psi(x,y)$ which says that $x$ is a set of ordinals which are all singular in $L$, $x$ is closed and of ordertype $\omega_1$; $y$ is a function from $x$ into $\omega$ such that $y(\alpha) \ne y(\beta)$ whenever $\alpha, \beta \in x$ and $\alpha \in \text{lim} \, C_{\beta}$.

We claim that if $\theta = sup (D) $ and  $D \subset \theta' < \omega_2$, and if $f: D \rightarrow$ is such that for  $\alpha , \beta \in D$, if $\alpha <_T \beta$ then $f(\alpha) \ne f (\beta)$, then $\theta'$ is a regular cardinal in $L$.

Indeed assume $\theta'$ were singular in $L$, then $C_{\theta'}$ would have been defined, so $D \cap \text{lim} \, C_{\theta'}$ is uncountable. In particular, there is $\alpha < \beta \in D$ such that $f(\alpha)= f(\beta)$, yet $\alpha<_T \beta$, which is a contradiction.
To summarize:
\begin{theorem}
For every $L$-regular cardinal $\kappa$, there is a proper forcing denoted by $\mathbb{T}^{\lambda}$ and a $\Sigma_1$-formula $ \phi$ such that
in the generic extension $V[\mathbb{T}^{\kappa}]$,
\[ V[\mathbb{T}^{\kappa}] \models \phi(\kappa) \Leftrightarrow \kappa \text{ is a regular cardinal in $L$ } \]

\end{theorem}

\subsection{Coding machinery}

When applied over $L$, the $\vec S$-coding forcings can be used to surgically alter the truth-value of certain projective formulas. These formulas and the $\vec S$-coding forcings will be used to form projective predicates $\Phi_{n,m}$, for $n,m \in \omega$, where each $\Phi_{n,m}$ is a $\Sigma^1_n$-formula which will eventually serve as our graph of our uniformizing functions for $A_{n,m}$, where $A_{n,m}$ is the $m$-th $\Sigma^1_n$-set in the plane.

We first describe how to use the $\vec S$-coding forcings to obtain suitable $\Sigma^1_n$-predicates, whose truth-value can be manipulated using the right forcings.

We work over $L$ as our ground model. Let $r \in L$ be a real. Note that over $L$ there  is a sequence $$\vec{S}=( S_ \alpha\mid \alpha\in \Lim)$$ uniformly definable satisfying
    \begin{itemize}
        \item $S_\alpha\subset \alpha$
        \item If $\alpha$ is a regular cardinal, then $S_\alpha$ is stationary co-stationary in $\alpha\cap \Cof(\omega)$.
    \end{itemize}
     The existence of such a sequence $( S_ \alpha\mid \alpha\in \Lim(\delta))$  follows from the fact that $\diamondsuit_{\lambda}$ holds in $L$ for any $L$-cardinal $\lambda$ and is a routine construction. 

 We aim to find first a $\Sigma^1_3$-formula $\Phi$ such that $L \models \lnot \exists x \Phi (x)$, yet for $\eta \in$ Ord there is a coding forcing Code$(r, \eta) \in L$ such that after forcing with Code$(r,\eta)$, in the resulting universe $L[\text{Code} (r,\eta)] \models \Phi (r)$ and $\forall  s (s \ne r \rightarrow \lnot \Phi(s) )$ does hold.

The desired forcing Code($r,\eta$) will be itself a three step iteration denoted by $\forceQ^0 (r,\eta) \ast \dot{\forceQ}^1(r,\eta) \ast \dot{\forceQ}^2 (r,\eta)$, where the first and second factor are iterations themselves, and we will describe it now.
Given our real $r$ and $\eta \in L $ such that $\eta$ is a limit cardinal in $L$, we first look at the $\omega$-block of $L$-cardinals which follow $\eta$, that is we form the interval $[\eta , (\eta^{+ \omega})^L )$.

The first factor of the iteration, denoted by $\forceQ^0 (r,\eta)$, will use a fully supported $\omega$-length iteration, with $\mathbb{T}^{\eta^{+i}}$ for $i \in \omega$ as factors. Consequentially, after forcing with this partial order, each $L$-cardinal in the interval $[\eta, \eta^{+\omega})$ is $\Sigma_1$-definable using $\omega_1$ as the only parameter in its definition.

 For the second factor of our coding forcing we first identify the real $r$ with the according subset of $\omega$ and form the second factor of $\text{Code} (r,\eta)$ as follows:
$$\forceQ^1(r,\eta):=\prod_{j\in r}Club(S_{\eta+2j}) \times \prod_{j\notin r}Club(S_{\eta+2j+1})$$ using countable support.
 
After the two forcings are done we define (in the resulting generic extension) $ X \subset \omega_1$ to be the $<$-least set  (in some previously fixed well-order of some sufficiently large $H(\theta)$) which codes the following objects:
\begin{itemize}
\item The $<$-least set of closed sets of $\eta^{+i}$ of ordertype $\omega_1$ and the according specializing functions, both added with the factors of the forcing of the form $\mathbb{T}^{\eta^{+i}}$. These sets together ensure the $\Sigma_1$-definability of the $L$-cardinals in the interval $[\eta, \eta^{+ \omega})$.
\item And $\omega$-many $\omega$-closed, unbounded subsets through some of the canonically definable stationary sets $S_{\eta^{+i}} \subset \eta^{+i} \cap \operatorname{cof} (\omega)$. These $\omega$-closed sets are so chosen that the
characteristic function of $r$ can be read off. That is  we collect $\{ c_{\eta^{+i}} \subset S_{\eta^{+i}} \, : \, i=  2n \land n \notin r \}$ and  $\{ c_{\eta^{+i}} \subset S_{\eta^{+i}} \, : \, i=  2n+1 \land n \in r \}$.
\end{itemize}

As mentioned already, when working in $L[X]$ then
 we can read off $r$  via looking at the $\omega$-block of $L$-cardinals starting at $\eta$ and determine which canonical $L$-stationary set $S_{\eta^{+i}}$ contains  an $\omega$-closed, unbounded set in $L[X]$:
\begin{itemize}
 \item[$(\ast)$]  $n \in r$ if and only if $S_{\eta^{+(2n+1)}}$ contains an $\omega$-closed unbounded set, and $n \notin r$ if and only if $S_{\eta^{+(2n)}}$ contains an $\omega$-closed unbounded set.
\end{itemize}
Indeed this follows readily from Lemma \ref{lem: stationary preservation}.

We note that we can apply an argument resembling David's trick in this situation. We rewrite the information of $X \subset \omega_1$ as a subset $Y \subset \omega_1$ using the following line of reasoning.
It is clear that any transitive, $\aleph_1$-sized model $M$ of $\ZFP$ which contains $X$ will be able to correctly decode out of $X$ all the information.
Consequentially, if we code the model $(M,\in)$ which contains $X$ as a set $X_M \subset \omega_1$, then for any uncountable $\beta$ such that $L_{\beta}[X_M] \models \ZFP$ and $X_M \in L_{\beta}[X_M]$:
\[L_{\beta}[X_M] \models ``\text{The model decoded out of }X_M \text{ satisfies $(\ast)"$.} \]
In particular there will be an $\aleph_1$-sized ordinal $\beta$ as above and we can fix a club $C \subset \omega_1$ and a sequence $(M_{\alpha} \, : \, \alpha \in C)$ of countable elementary submodels  of $L_{\beta} [X_M]$ such that
\[\forall \alpha \in C (M_{\alpha} \prec L_{\beta}[X_M] \land M_{\alpha} \cap \omega_1 = \alpha)\]
Now let the set $Y\subset \omega_1$ code the pair $(C, X_M)$ such that the odd entries of $Y$ should code $X_M$ and if $Y_0:=E(Y)$ where the latter is the set of even entries of $Y$ and $\{c_{\alpha} \, : \, \alpha < \omega_1\}$ is the enumeration of $C$ then
\begin{enumerate}
\item $E(Y) \cap \omega$ codes a well-ordering of type $c_0$.
\item $E(Y) \cap [\omega, c_0) = \emptyset$.
\item For all $\beta$, $E(Y) \cap [c_{\beta}, c_{\beta} + \omega)$ codes a well-ordering of type $c_{\beta+1}$.
\item For all $\beta$, $E(Y) \cap [c_{\beta}+\omega, c_{\beta+1})= \emptyset$.
\end{enumerate}
We obtain
\begin{itemize}
\item[$({\ast}{\ast})$] For any countable transitive model $M$ of $\ZFP$ such that $\omega_1^M=(\omega_1^L)^M$ and $ Y \cap \omega_1^M \in M$, $M$ can construct its version of the universe $L[Y \cap \omega_1^N]$, and the latter will see that there is an $\aleph_1^M$-sized transitive model $N \in L[Y \cap \omega_1^N]$ which models $(\ast)$ for $r$ and $\eta$.
\end{itemize}
Thus we have a local version of the property $(\ast)$.

We know define the third, and last factor of $\operatorname{Code} (r,\eta)$, working in $L[\forceQ^0(r,\eta)] [\dot{\forceQ}^1(r,\eta) ]$ we shall define $\forceQ^2 (r,\eta)$ as follows. We use almost disjoint forcing $\mathbb{A}_D(Y)$ relative to our previously defined, almost disjoint family of reals $D \in  L $ (see the paragraph after Definition 2.5)  to code the set $Y\subset \omega_1$ into one real $r$. This forcing only depends on the subset of $\omega_1$ we code, thus $\mathbb{A}_D(Y)$ will be independent of the surrounding universe in which we define it, as long as it has the right $\omega_1$ and contains the set $Y$.

We finally obtained a real $R$ such that
\begin{itemize}
\item[$({\ast}{\ast}{\ast})$] For any countable, transitive model $M$ of $\ZFP$ such that $\omega_1^M=(\omega_1^L)^M$ and $ R  \in M$, $M$ can construct its version of $L[R]$ which in turn thinks that there is a transitive $\ZFP$-model $N$ of size $\aleph_1^M$  such that $N$ believes $(\ast)$ for $r$ and $\eta$.
\end{itemize}
Note that $({\ast} {\ast} {\ast})$ is a $\Pi^1_2$-formula in the parameters $R$ and $r$. We will often suppress the sets$r,\eta $ when referring to $({\ast} {\ast} {\ast})$ as they will be clear from the context. We say in the above situation that the real $r$ \emph{ is written into $\vec{S}$}, or that $r$ \emph{is coded into} $\vec{S}$ (at $\eta$) and $R$ witnesses that $r$ is coded (at $\eta$). 
 
 The projective and local statement $({\ast} { \ast} {\ast} )$, if true,  will determine how certain inner models of the surrounding universe will look like with respect to branches through $\vec{S}$.
That is to say, if we assume that $({\ast} { \ast} {\ast} )$ holds for a real $r$ and is the truth of it is witnessed by a real $R$. Then $R$ also witnesses the truth of $({\ast} { \ast} {\ast} )$ for any transitive $\ZFP$-model $M$ which contains $R$ (i.e. we can drop the assumption on the countability of $M$).
Indeed if we assume 
that there would be an uncountable, transitive $M$, $R \in M$, which witnesses that $({\ast} { \ast} {\ast} )$ is false. Then by L\"owenheim-Skolem, there would be a countable $N\prec M$, $R \in N$ which we can transitively collapse to obtain the transitive $\bar{N}$. But $\bar{N}$ would witness that $({\ast} { \ast} {\ast} )$ is not true for every countable, transitive model, which is a contradiction.

Consequentially, the real $R$ carries enough information that
the universe $L[R]$ will see that certain $L$-stationary sets from $\vec{S}$ have clubs in that
\begin{align*}
n \in r \Rightarrow L[R] \models  ``S_{\eta^{+ (2n+1)}} \text{ contains an $\omega$-closed unbounded set}".
\end{align*}
and
\begin{align*}
n \notin r\Rightarrow L[R] \models ``S_{\eta^{ + (2n)} }\text{ contains an $\omega$-closed unbounded set}".
\end{align*}
Indeed, the universe $L[R]$ will see that there is a transitive $\ZFP$-model $N$ which believes $(\ast)$. The latter being coded into $R$. But by upwards $\Sigma_1$-absoluteness, and the fact that $N$ can compute $\vec{S}$ correctly, if $N$ thinks that some $L$-stationary set in $\vec{S}$ contains an $\omega$-closed, unbounded set, then $L[R]$ must think so as well.

\section{Suitable $\Sigma^1_n$-predicates}
We shall use the $\Sigma^1_3$-predicate ``being coded into $\vec{S}"$ (we will often write just ``being coded$"$ for the latter) to form suitable $\Sigma^1_n$-predicates $\Phi^n$ for every $n \in \omega$. These predicates share the following properties:
\begin{enumerate}
\item $L \models \forall x  \lnot(\Phi^n (x))$
\item For every real $x \in L$, there is a coding forcing $\operatorname{Code}^n (x) \in L$ such that after forcing with it, $L[\operatorname{Code}^n (x)] \models \Phi^n (x)$, and for every real $y \ne x$, $L[\operatorname{Code}^n (x)] \models \lnot \Phi^n(y)$.
\end{enumerate}
Most importantly, these properties remain true even when iterating the coding forcings
$\operatorname{Code}^n (x_i)$ for a sequence of (names of) reals.

The predicates $\Phi^n(x)$ will be defined now. 
\begin{itemize}
\item $\Phi^3 (x,y,m) \equiv \exists a_0 ( (x,y,m,a_0)$ is coded into $\vec{S})$.
\item $\Phi^4(x,y,m) \equiv \exists a_0 \forall a_1 ( (x,y,m,a_0,a_1)$ is not coded into $\vec{S} )$.
\item $\Phi^5(x,y,m) \equiv \exists a_0 \forall a_1 \exists a_2 ( (x,y,m,a_0, a_1,a_2)$ is coded into $\vec{S} )$.
\item $\Phi^6 (x,y,m) \equiv \exists a_0 \forall a_1 \exists a_2 a_3 ( (x,y,m,a_0,a_1, a_2,a_3)$ is not coded into $\vec{S})$.
\item ...
\item ...
\item $\Phi^{2n} (x,y,m) \equiv \exists a_0 \forall a_1...\forall a_{2n} (( x,y,m,a_0,...,a_{2n} )$ is not coded into $\vec{S} )$.
\item $\Phi^{2n+1} (x,y,m) \equiv \exists a_0 \forall a_{1}...\exists a_{2n+1} (( x,y,m,a_0,...,a_{2n+1})$ is coded into $\vec{S} )$.
\item ...
\item...

\end{itemize}
Each predicate $\Phi^n$ is exactly $\Sigma^1_n$.
In the choice of our $\Sigma^1_n$-formulas $\Phi^n(x)$, we encounter again a periodicity phenomenon, that is two different cases depending on $n\in \omega$ being even or odd, a theme which is pervasive in this area.
It is clear that for each predicate $\Phi^n$ and each given real $x$ there is a way to create a universe in which $\Phi^n (x)$ becomes true using our coding forcings. We just need to iterate the relevant coding forcings using countable support.
\begin{lemma}
Let $n \in \omega$ and let $x$ be a real in our ground model $L$. Then there is a forcing $\operatorname{Code}^n (x)$ such that if $G \subset \operatorname{Code}^n (x)$ is generic, the generic extension
$L[G]$ will satisfy $\Phi^n(x)$ and for every $y \ne x$, $L[G] \models \lnot \Phi^n(y)$.
This property can be iterated, that is it remains true if we replace $L$ with $L[G]$ in the above.

\end{lemma}
\section{Based forcing}

This section introduces a suitable collection of forcings, dubbed based forcings, which we use for the proof. As we proceed in our iteration, we simultaneously refine based forcings, via adding more and more constraints, yielding a decreasing hierarchy of $\alpha$-based forcings.

As a first step, we shall consider the problem of forcing the $\Sigma^1_3$-uniformization property over universes which satisfy the anti-large cardinal hypothesis that $\omega_1= \omega_1^L$ (or more generally satisfy that $\omega_1$ is not inaccessible to reals). Recall our $\Sigma^1_3$-predicate $\Phi^3(x)$ defined in the last two sections, and the attached coding forcing $\operatorname{Code} (x)$.
There is a very easy strategy to force the $\Sigma^1_3$-uniformization property for universes $V$ which satisfy $\omega_1=\omega_1^L$ provided $V$ thinks that $\Phi^3(x)$ does not hold for every real $x \in V$: we pick some bookkeeping function $F \in V$ which should list all the (names of) reals in our iteration and at each stage $\beta$, under the assumption that we created already $V[G_{\beta}]$, if $F(\beta)$ lists a real $x\in V[G_{\beta}]$ and a natural number $k$, we ask whether there is a real $y$ such that
\[ V[G_{\beta}] \models \varphi_k (x,y) \]
holds, where $\varphi_k$ is the $k$-th $\Sigma^1_3$-formula. If so, we pick the least such $y$ (in some fixed well-order), and let the value of the desired $\Sigma^1_3$-uniformizing $f_k(x)$ to be $y$. Additionally we force with $\operatorname{Code} (x,y,k)$ over $V[G_{\beta}]$ and obtain $V[G_{\beta+1}]$. 
The resulting universe $V[G_{\beta+1}]$ will satisfy that $\Phi^3(x,y,k)$ is true, whereas $\Phi^3(x,y',k)$ is not true for each $y' \ne y$. Moreover, because of upwards absoluteness of $\Sigma^1_3$-formulas, this property will remain true in all further generic extensions we create.

If we iterate long enough in order to catch our tail, the final model $V[G]$ will 
satisfy the $\Sigma^1_3$-uniformization property via
\[f_k (x)=y \Leftrightarrow \Phi^3 (x,y,k). \]
To summarize, the easy strategy to force $\Sigma^1_3$-uniformization is to consider at each step some $x$-section of some $\Sigma^1_3$-set $A_k \subset \omega^{\omega} \times \omega^{\omega}$, and if non-empty, pick the least $y$ for which $A_k(x,y)$ is true. Then force to make $\Phi^3(x,y,k)$ true and repeat.

Based forcings use this strategy for $\Sigma^1_3$-sets, while putting no constraints on $\Sigma^1_n$-sets for $n > 3$. Anticipating that we aim for a universe of $\BPFA$, we let $\kappa$ be the least reflecting cardinal in $L$.

\begin{definition}
Let $ \lambda < \kappa$ and let $F: \lambda \rightarrow H(\kappa)$ be a bookkeeping function. We say that an iteration $(\forceP_{\beta}, \dot{\forceQ}_{\beta} \mid \beta < \lambda)$ is (0-)based with respect to $F$ if the iteration is defined inductively via the following rules:
Assume that $F(\beta)= (\dot{x}, 3, \dot{k})$ where $\dot{x}$ is a $\forceP_{\beta}$-name of a real
and $\dot{k}$ a $\forceP_{\beta}$-name of a natural number. Also assume that $V[G_{\beta}] \models \exists y (\varphi_{3,k}(x,y))$ and let $\dot{y}$ be the $<$-least name of such a real in some fixed well-order of $H(\kappa)$.
Then let $\dot{\forceQ}_{\kappa}^{G_{\beta}}$ be $\operatorname{Code} (x,y,3,k)$.
If $V[G_{\beta}] \models \exists y (\phi_{3,k}(x,y))$ is not true use the trivial forcing.

Otherwise, we let the bookkeeping decide whether it provides us with a some reals $y,a_0,a_1,...,a_n$ and whether a  tuple of the form $(x,y,a_0,...,a_n,m,k)$ is coded or not.

\end{definition}

\subsection{Strategy to obtain global $\Sigma$-uniformization}
We shall describe the underlying idea to force the global $\Sigma$-uniformization.
The definition of the factors of the iteration will depend on whether the formula
$\varphi_m$ we consider at our current stage is in $\Sigma^1_{2n}$ or in $\Sigma^1_{2n+1}$ where $n\ge 2$. We start with the case where $\varphi_m$ appears first on an odd level of the projective hierarchy.
Assume that $\varphi_m \equiv \exists a_0 \forall a_1 \exists a_2 ... \exists a_{2n-2} \psi( x,y,a_0,a_1,... a_{2n} )$  is a $\Sigma^1_{2n+1}$-formula (where $\psi(x,y,a_0,...a_{2n})$ is a $\Pi^1_2$-formula quantifying over the two remaining variables $a_{2n-1}$ and $a_{2n}$) and  $x$ is a real. We want to find a value for the uniformization function for $\varphi_m$ at the $x$-section.

To start, we list all triples of reals $ ((x,y^0,a_0^0) , (x, y^1, a_0^1), (x,y^2,a_0^2) ,...)$ according to some fixed well-order $<$.  There will be a $<$-least triple $(x,y^{\alpha}, a_0^{\alpha} )$ for which $\forall a_1 \exists a_2... \exists a_{2n-2} \psi (x,y^{\alpha},a^{\alpha}_0, a_1,a_2...)$ is true, otherwise the $x$-section would be empty and there is nothing to uniformize.

The goal will be to set up the iteration in such a way that all triples $(x,y^{\beta} ,a_0^{\beta})$, $\beta > \alpha$ will satisfy the following formula, which is $\Pi^1_{2n+1}$ in the parameters $( x , y^{\beta}, a^{\beta}_0 ,m)$ as is readily checked:
\begin{align*}
\forall a_{1} \exists a_{2} \forall a_{3}... \exists a_{2n-2} (( x,y^{\beta} ,m, a_0^{\beta},a_1,a_2...,a_{2n-2} )\text{ is not coded into $\vec{S} )$.}
\end{align*}
At the same time the definition of the iteration will ensure that $(x,y^{\alpha}, a^{\alpha}_0)$ will satisfy 
\begin{align*}
\exists a_{1} \forall a_{2} \exists a_{3}... \forall a_{2n-2} (( x,y^{\alpha} ,m, a_0^{\alpha},a_1,a_2...,a_{2n-2} )\text{ is  coded into $\vec{S} )$.}
\end{align*}
Provided we succeed, the pair $(x,y^{\alpha})$ will then be the unique solution to the following formula, which is $\Sigma^1_{2n+1}$, and which shall be the defining formula for our uniformizing function:
\begin{align*}
\sigma_{\text{odd}}(x,y,m) \equiv \exists a_0 ( \forall a_1 \exists a_2...\psi (x,y,a_0,a_1,...) \land \\
\lnot ( \forall a_1 \exists a_2... ( ( x,y ,m, a_0 ,a_1 ,a_2...) \text{ is not coded into $\vec{S}$}) )
\end{align*}
Indeed, for all $\beta> \alpha$, $(x,y^{\beta},a^{\beta}_0)$ can not satisfy the second subformula of $\Psi$ whereas for all
$\beta< \alpha$, $(x,y^{\beta} , a_0^{\beta}$ can not satisfy the first subformula, as $(x,y^{\alpha}, a_0^{\alpha} )$ is the least such triple.

If we assume that $\varphi_m$ is on an even level of the projective hierarchy we will define things in a dual way to the odd case. 
Assume that $\varphi_m \equiv \exists a_0 \forall a_1 \exists a_2 ... \forall a_{2n-3} \psi( x,y,a_0,a_1,... a_{2n-1} )$ is a $\Sigma^1_{2n}$-formula and  $x$ is a real. We want to find a value for the uniformization function for $\varphi_m$ at the $x$-section.

Again, we list all triples of reals $ ((x,y^0,a_0^0) , (x, y^1, a_0^1), (x,y^2,a_0^2) ,...)$ according to some fixed well-order $<$.  There will be a $<$-least triple $(x,y^{\alpha}, a_0^{\alpha} )$ for which $\forall a_1 \exists a_2... \forall a_{2n-3} \psi (x,y^{\alpha},a^{\alpha}_0, a_1,a_2...)$ is true, otherwise the $x$-section would be empty and there is nothing to uniformize.

The goal will be to set up the iteration in such a way that all triples $(x,y^{\beta} ,a_0^{\beta})$, $\beta > \alpha$ will satisfy the following formula, which is $\Pi^1_{2n}$ in the parameters $( x , y^{\beta}, a^{\beta}_0,m)$ as is readily checked:
\begin{align*}
\forall a_{1} \exists a_{2} \forall a_{3}... \exists a_{2n-3} (( x,y^{\beta} ,m, a_0^{\beta},a_1,a_2...,a_{2n-3} )\text{ is coded into $\vec{S} )$.}
\end{align*}
At the same time the definition of the iteration will ensure that $(x,y^{\alpha}, a^{\alpha}_0)$ will satisfy 
\begin{align*}
\exists a_{1} \forall a_{2} \exists a_{3}... \exists a_{2n-3} (( x,y^{\alpha} ,m, a_0^{\alpha},a_1,a_2...,a_{2n-3} )\text{ is not  coded into $\vec{S} )$.}
\end{align*}
Provided we succeed, the pair $(x,y^{\alpha})$ will then be the unique solution to the following formula, which is $\Sigma^1_{2n}$, and which shall be the defining formula for our uniformizing function:
\begin{align*}
\sigma_{\text{even}} (x,y,m) \equiv \exists a_0 ( \forall a_1 \exists a_2...\psi (x,y,a_0,a_1,...) \land \\
\lnot ( \forall a_1 \exists a_2... ( ( x,y ,m, a_0 ,a_1 ,a_2...) \text{ is coded into $\vec{S}$}) )
\end{align*}

\subsection{Definition of the iteration according to the strategy.}

We shall elaborate on the ideas presented in the last section. Our goal is to define an iteration such that our formulas $\sigma_{\text{even}}$ and $\sigma_{\text{odd}}$ work as defining formulas for the uniformization functions.
Later we will combine these ideas with the standard iteration to obtain $\BPFA$ from a reflecting cardinal to finally finish the proof of the main theorem.

We will need the following predicates. For an arbitrary real $x$, list the triples
$(x,y^0,a_0^0), (x,y^1,a_0^1), (x,y^2,a_0^2),...)$ according to our fixed well-order $<$.
For each ordinal $\alpha< 2^{\aleph_0}$ we fix
bijections $\pi_{\alpha}: (2^{\aleph_0})^{\alpha} \rightarrow 2^{\aleph_0}$ (we assume w.l.o.g that such a bijection exist as we always force its existence with a proper forcing).

We say \[\Psi( \#\psi, x,y^{\alpha},a_0^{\alpha},b_1,...,b_{n} ) \text{ is true } \] if
for every $\beta  < \alpha$, if we let $\pi_{\alpha}^{-1} (b^k)= (c^k_{\beta})_{\beta < \alpha}$ then
$\psi( x,y^{\beta}, a_0^{\beta}, c^1_{\beta}, c^2_{\beta},...,c^{n}_{\beta})$ is true, where $\psi$ is a $\Pi^1_{2}$-formula.
Note that strictly speaking, $\Psi$ is not first order as it is a conjunction of $\alpha$-many formulas. We will use coding to create a projective predicate which is true whenever $\Psi$ is true.

Let $F : \kappa \rightarrow H(\lambda)$ be some bookkeeping function in $L$ which shall be our ground model and which guides our iteration.
We assume that we have defined already the following list of notions:
\begin{itemize}
\item We do have the notion of $\alpha_{\beta}$-based forcings, which is a subset of the set of based forcings. There is a set of tuples of reals which
is tied to the $\alpha_{\beta}$-based forcings and which consists of tuples of reals we must not use in coding forcings anymore as factors of an $\alpha_{\beta}$-based forcing.

\item We defined already our iteration $\forceP_{\beta} \in L$ up to stage $\beta$.
\item We picked a $\forceP_{\beta}$-generic filter $G_{\beta}$ for $\forceP_{\beta}$ and work, as usual, over $L[G_{\beta}]$.
\end{itemize}

The bookkeeping function $F$ at $\beta$ hands us an $ \omega$-tuple $(\dot{x},\dot{m}, \dot{l}, \dot{b}_0,\dot{b}_1,\dot{b}_2,...)$ where $\dot{x}$ and $\dot{b}_n$ are $\forceP_{\beta}$-names of a reals and $\dot{m}$ and $\dot{l}$ are both $\forceP_{\beta}$-names of natural numbers (in fact at each stage we will only need finitely many of those names of reals, and the rest of the information will be discarded).
We let $x=\dot{x}^{G_{\beta}}$, and define $b_n,m,l$ accordingly. 
Our goal is to define the forcing $\dot{\forceQ}_{\beta}$ we want to use at stage $\beta$, and to define the notion of $\alpha_{\beta}+1$-allowable forcing.
We consider various cases for our definition.

\subsubsection{Case 1}

In the first case, writing $\varphi_m= \exists a_0 \forall a_1... \exists a_{2n-2} \psi (x,y, a_0,...,a_{2n-2} )$, where $\psi$ is a $\Pi^1_2$-formula (and $\varphi_m$ is a 
$\Sigma^1_{2n-1}$ formula for $n \ge 3$)   we assume that the bookkeeping $F$ at $\beta$ hands us
a tuple $c_1,...c_{2n-2}$ of real numbers, $x$, and $\alpha< 2^{\aleph_0}$ with the assigned reals $a_0^{\alpha}$ and $y^{\alpha}$.

Recall that we have our fixed bijection $\pi_{\alpha}:( 2^{\aleph_0})^{\alpha} \rightarrow 2^{\aleph_0}$ (note that such a bijection exists at least in a proper generic extension, so we pass to that universe if necessary)
We assume that
\begin{align*}
\psi (x,y^{}, a^{}_0 ,c_1,...,c_{2n-2}) \text{ is true}. 
\end{align*}

We first define $\alpha_{\beta}+1$-based forcings as just $\alpha_{\beta}$-based with the additional constraints that we must not use any forcings of the form
\begin{align*}
\operatorname{Code} (\# \psi,x,y^{i}, a^{i}_0,e_1,e_2,...,e_{2n-2} )
\end{align*}
as a factor of our iteration, where $i > \alpha$ and $e_1,e_2,...,e_{2n-2}$ are such that
there are reals $r^j_0,r^j_1,...,r^j_{i}$ with $r^j_{\alpha}= c_j$ and
$\pi_i (r^j_0,...,r^j_{i} ) = e_j$ for every $j \in [1,2n-2].$

We use the coding forcing
\begin{align*}
\operatorname{Code} (\# \psi,x,y^{\alpha},a_0^{\alpha}, g_1,g_2,...,g_{2n-2}) =: \dot{\forceQ}_{\beta}
\end{align*}
as the $\beta$-th factor of our iteration, where $g_1,...,g_{2n-2}$ are reals which we are given by the bookkeeping and such that $\Psi (\# \psi,x,y^{\alpha},a_0^{\alpha}, g_1,...,g_{2n-2} )$ is true.

If $\Psi (\# \psi,x,y^{\alpha},a_0^{\alpha}, g_1,...,g_{2n-2} )$ is  not true, then pick the $<$-least tuple $g_1,...,$ $g_{2n-2}$ ) such that $\Psi (\# \psi,x,y^{\alpha},a_0^{\alpha}, g_1,...,g_{2n-2} )$ is true. If there is no such tuple then do nothing.

The upshot of this definition is the following. Suppose we iterate following our rules, and in the final model $L[G_{\kappa}]$, $\exists a_0 \forall a_1 ...\exists a_{2n-2} \psi( x,y,a_0,a_1,...,a_{2n-2} )$ is true and $(x,y,a_0)$ is the $<$-least triple witnessing this.
Then also
\[\exists a_0 \forall a_1 ... \exists a_{2n-2} ((\# \psi,x,y, a_0,a_1,...,a_{2n-2})\text{ is coded) } \] is true in $L[G_{\kappa}]$ by the rules of the iteration.
At the same time for every $(x,y'',a''_0) > (x,y,a_0)$ and for every real $c_1$ there will be a real $b'_2$ such that for every $c_3$ there is a further $b'_4$ and so on such that
$(x,y'',a''_0, \# \psi, c_1,b'_2,...)$ is not coded.

\subsubsection{Case 2}
We again let $\varphi_m= \exists a_0 \forall a_1... \exists a_{2n-2} \psi (x,y, a_0,...,a_{2n-2} )$  and we assume that the bookkeeping $F$ at $\beta$ hands us
a tuple $c_1,...c_{2n-2}$ of real numbers, $x$, and $\alpha< 2^{\aleph_0}$ with the assigned reals $a_0^{\alpha}$ and $y^{\alpha}$.

We assume that

\begin{align*}
\psi (x,y^{}, a^{}_0 ,c_1,...,c_{2n-2}) \text{ is not  true}. 
\end{align*}
In this situation we force with the trivial forcing at the $\beta$-th stage and do not define the notion of $\alpha_{\beta}+1$-based forcing. Instead we continue working with $\alpha_{\beta}$-based forcings.

\subsubsection{Case 3}
This is the dual case to the first one. We assume that
\[\varphi_m \equiv \exists a_0 \forall a_1,..., \forall a_{2n} (\psi (x,y,a_0,...,a_{2n-3} ))\] that is $\varphi_m$ belongs to an even projective level.
We assume that the bookkeeping $F$ at $\beta$ hands us
a tuple $c_1,...c_{2n-3}$ of real numbers, $x$, and $\alpha< 2^{\aleph_0}$ with the assigned reals $a_0^{\alpha}$ and $y^{\alpha}$.

We assume that
\begin{align*}
\psi (x,y^{\alpha}, a^{\alpha}_0 ,c_1,...,c_{2n-3}) \text{ is true}. 
\end{align*}

First we define the $\beta$-the factor of our iteration. We force with
\begin{align*}
\operatorname{Code} (\# \psi,x,y^{i}, a^{i}_0,e_1,e_2,...,e_{2n-3} )=: \dot{\forceQ}_{\beta}
\end{align*} where $i > \alpha$ and $e_1,e_2,...,e_{2n-3}$ are reals given by the bookkeeping which have the property  that
there are reals $r^j_0, r^j_1,...,r^j_{i}$ with $r^j_{\alpha}= c_j$ and
$\pi_i (r^j_0,...,r^j_{i} ) = e_j$ for every $j \in [1,2n-3].$

Next we define the notion of $\alpha_{\beta} +1 $-based which should be $\alpha_{\beta}$-based and the additional constraint that
\begin{align*}
\operatorname{Code} (\# \psi,x,y, \alpha, g_1,g_2,...,g_{2n-3}) =: \dot{\forceQ}_{\beta}
\end{align*}
must not be used as a factor of any future forcing we use in our iteration. Here $g_1,...,g_{2n-3}$ are reals which we are given by the bookkeeping with the property that $\Psi (\# \psi, x,y^{\alpha}, a_0^{\alpha}, g_1,...,g_{2n-3}) $ (see at the beginning of section 2.2 for a definition of $\Psi$) holds true. 

\subsubsection{Case 4}
This is the same as in case 2. We find ourselves in the situation where $\psi(x,y^{\alpha},a_0^{\alpha},c_1,...,c_{2n-3})$ is not true. Then we do nothing.

This ends the definition of our iteration. We shall argue later, that the definition works in that it will produce a universe where $\sigma_{\text{even}}$ and $\sigma_{\text{odd}}$ serve as definitions of uniformizing functions, when applied in the right context.

\begin{definition}
Let $F: \kappa \rightarrow H(\lambda)$ be a bookkeeping function. We say that an iteration $(\forceP_{\beta} ,\dot{\forceQ}_{\beta} \mid \beta < \kappa)$ is suitable if each factor of the iteration either is proper or each factor is obtained using $F$ and the four rules described above.
\end{definition}

\section{Proving the main theorem}

Recall first that $\kappa$ is reflecting whenever it is regular and such that $V_{\kappa}$ is $\Sigma_2$-elementary in the universe $V$. We let $\kappa$ be the least reflecting cardinal in $L$. We let $F: \kappa \rightarrow H(\kappa)$ be such that every element of $H(\kappa)$ has an unbounded pre-image under $F$.
We assume that we arrived at stage $\beta < \kappa$ of our iteration and will describe what to force with now.

\subsection{Definition of the iteration. Odd stages}

Assume that $F(\beta)$ hands us names such that they evaluate with the help of our current generic $G_{\beta}$ to a tuple $(\#\psi,x,y^{\alpha}, a_0^{\alpha}, c_1,...,c_n,...,)$ where $\# \psi$ is the the G\"odelnumber of a $\Sigma^1_{n}$-formula and  $x,y^{\alpha}, a_0^{\alpha}, c_1,...)$ are reals.

In this situation we just follow the one of the four cases in the definition of suitable forcings which applies.

\subsection{Definition of the iteration. Even stages.}
We work towards $\BPFA$ on the even stages of the iteration. Recall that due to J. Bagaria (see \cite{Bag}) $\BPFA$ is equivalent to the assertion that $H(\omega_2) \prec_{\Sigma_1} V^{\forceP}$ for any proper forcing $\forceP$. We shall work towards Bagaria's reformulation of $\BPFA$ adapting the usual way.

Assume that the bookkeeping $F$ at stage $\beta$ hands us a $\Sigma_1$-formula $\sigma$ and parameters $p_0,...,p_l$ of $\sigma$ which are elements of $P(\omega_1) \cap L[G_{\beta} ]$. We ask of $\sigma(p_0,...,p_l)$ whether it holds in a suitable generic extension of $L[G_{\beta}]$. If yes, then as $\kappa$ is reflecting, there is such a forcing already in $L_{\kappa} [G_{\beta}]$.  Moreover the witness to the true $\Sigma_1$-formula can be assumed to have a name in $L_{\kappa} [G_{\beta}]$. We fix a $\gamma < \kappa$ such that $L_{\gamma}$ is $\Sigma_2$-elementary in $L_{\kappa}$, which implies that also $L_{\gamma} [G_{\beta}]$ is  $\Sigma_2$-elementary in $L_{\kappa} [G_{\beta}]$. Hence the suitable forcing$\forceQ$ which forces $\sigma(p_0,...,p_l)$ to become true can be assumed to belong to $L_{\gamma} [G_{\beta}]$ already. We force with $\forceQ$ at stage $\beta$ then.

\subsection{Properties of the defined universe}

\begin{lemma}
Let $G_{\kappa}$ denote a generic filter for the full, $\kappa$-length iteration.
\begin{enumerate}

\item
Then $L[G_{\kappa}]$ satisfies that
whenever $\varphi_m= \exists a_0 \forall a_1...\exists a_{2n-2} \psi(x,y,a_0,...,a_{2n-2})$ and $(x,y^{\alpha},a_0^{\alpha})$ is such that
\[ L[G_{\kappa}] \models \varphi_m (x,y^{\alpha},a_0^{\alpha}) \]
Then for each $\beta > \alpha$
\[ L[G_{\kappa}] \models \forall a_1 \exists a_2...\exists a_{2n-2} ( (\#\psi,x,y^{\beta} ,a_0^{\beta}, a_1,..,a_{2n-2} ) 
\text{ is not coded }) \]
\item If $\varphi_m= \exists a_0 \forall a_1...\exists a_{2n-3} \psi(x,y,a_0,...,a_{2n-3})$ and $(x,y^{\alpha},a_0^{\alpha})$ is such that
\[ L[G_{\kappa}] \models \varphi_m (x,y^{\alpha},a_0^{\alpha}) \]
Then for each $\beta > \alpha$
\[ L[G_{\kappa}] \models \forall a_1 \exists a_2...\exists a_{2n-3} ( (\#\psi, x,y^{\beta} ,a_0^{\beta}, a_1,..,a_{2n-3} ) 
\text{ is coded }) \]
\end{enumerate}

\end{lemma}
\begin{proof}
We will only proof the first item, as the second is shown in the dual way.
Let $a_1$ be an arbitrary real.
As $L[G_{\kappa}] \models \varphi_m (x,y^{\alpha},a_0^{\alpha})$, we know that there is an $a_2$ such that for every $a_3$ and so on there is an $a_{2n-2}$ such that
$L[G_{\kappa}] \models \psi (x,y^{\alpha},a_0^{\alpha},a_1,...,a_{2n-2} )$.
By case 1 in the definition of the iteration, this translates to
$L[G_{\kappa}] \models `` (\#\psi,x,y^{\beta}, a_0^{\beta}, a_1,...,a_{2n-2})$ is not coded$"$.
As $a_1$ was arbitrary we can conclude that
\[L[G_{\kappa} ] \models \forall a_1 \exists a_2...\exists a_{2n-2} ( (\#\psi,x,y^{\beta} ,a_0^{\beta}, a_1,..,a_{2n-2} ) 
\text{ is not coded }) \] as desired.

\end{proof}

\begin{lemma}
Again $G_{\kappa}$ denotes a generic filter for the entire, $\kappa$-length iteration.
\begin{enumerate}
\item If $\alpha$ is least such that
\[L[G_{\kappa}] \models \forall a_1...\exists a_{2n-2} \psi(x,y^{\alpha},a^{\alpha}_0,...,a_{2n-2}),\]
then 
\[ L[G_{\kappa}] \models \exists a_1 \forall a_2 ...\forall a_{2n-2} ( (\# \psi, x,y^{\alpha},a_0^{\alpha},a_1,a_2,...,a_{2n-2} ) \text{ is coded}.) \]
\item If $\alpha$ is least such that
\[L[G_{\kappa}] \models \forall a_1...\exists a_{2n-3} \psi(x,y^{\alpha},a^{\alpha}_0,...,a_{2n-3}),\]
then 
\[ L[G_{\kappa}] \models \exists a_1 \forall a_2 ...\forall a_{2n-3} ( (\# \psi, x,y^{\alpha},a_0^{\alpha},a_1,a_2,...,a_{2n-3} ) \text{ is not coded}.) \]
\end{enumerate}
\end{lemma}
\begin{proof}
As $\alpha$ is the first ordinal for which \[L[G_{\kappa}] \models \forall a_1...\exists a_{2n-2} \psi(x,y^{\alpha},a^{\alpha}_0,...,a_{2n-2}),\] we know that
$\forall \beta < \alpha$, $L[G_{\kappa}] \models \exists a_1 \forall a_2,...,\forall a_{2n-2} (\lnot \psi (x,y^{\beta},a_0^{\beta}, a_1,...))$.
In particular, for every $\beta < \alpha$ there are reals $a_1^{\beta}$ such that for every $a_2^{\beta}$ there are reals $a_3^{\beta}$ and so on such that $\lnot \psi$ holds. Using the bijection $\pi_{\alpha}$, we can find an $a_1= \pi_{\alpha} ( (a_1^{\beta})_{\beta < \alpha} )$ which has the property that for any real $a_2$ there will be a real $a_3 = \pi_{\alpha} (a_3^{\beta})_{\beta <\alpha} )$ such that for any real $a_4$ and so on $\lnot \psi$ is true. 

But this translates via the case 1 rule of our iteration to the assertion that
\[ L[G_{\kappa}] \models \exists a_1 \forall a_2 ...\forall a_{2n-2} ( (\# \psi, x,y^{\alpha},a_0^{\alpha},a_1,a_2,...,a_{2n-2} ) \text{ is coded}.)\]

\end{proof}

\begin{lemma}
In $L[G_{\kappa}]$ the $\Sigma^1_{n+1}$-uniformization property holds true for every $n \ge 1$.
\end{lemma}
\begin{proof}
Let $\varphi \equiv \exists a_0 \forall a_1 ... (\psi (x,y,a_0,a_1,...,a_{2n-2}))$ be an arbitrary $\Sigma^1_{2n+1}$-formula in two free variables where $\psi$ is $\Pi^1_2$.
Let $x$ be a real such that there is a real $y $ with $L[G_{\kappa}] \models \varphi(x,y)$. We list all the triples $(x,y^{\alpha} ,a_0^{\alpha})$ according to our well-order $<$.
Let $\alpha$ be least such that
\[ L[G_{\kappa}] \models \forall a_1... (\psi (x,y^{\alpha},a_0^{\alpha},a_1,...) ).\]
Then by the last Lemma
\[ L[G_{\kappa}] \models \exists a_1 \forall a_2... \forall a_{2n-2} ( (\# \psi,x,y^{\alpha},a^{\alpha}_0,...) \text{ is coded} ) \]
holds true. Note that this formula is $\Sigma^1_{2n+1}$.
By the penultimate Lemma, for each $\beta > \alpha$
\[ L[G_{\kappa}] \models\forall a_1 \exists a_2... \exists a_{2n-2} ((\#\psi,x,y^{\beta},a^{\beta}_0,a_1,..) \text{ is not coded}), \]
which is $\Pi^1_{2n+1}$

So $(x,y^{\alpha})$ is the unique pair satisfying the $\Sigma^1_{2n+1}$-formula
\begin{align*}
\exists a_0 ( (\forall a_1 \exists a_2...\exists a_{2n-2} ( \psi (x,y,a_0,...) \land \\
\lnot (\forall a_1 \exists a_2 ...\exists a_{2n-2} ((\#\psi,x,y,a_0,a_1,...a_{2n-2} \text{ is not coded} ) )
\end{align*}
Indeed

\end{proof}

\begin{lemma}
$L[G_{\kappa}] \models \BPFA$.

\end{lemma}
\begin{proof}
This is standard. We show that $H(\omega_2)^{L[G_{\kappa}]}$ is $\Sigma_1$-elementary in any $L[G_{\kappa}]^{\forceP}$, where $\forceP$ is proper. We assume that the $\Sigma_1$-formula $\tau( x,p_0,...,p_n)$ can be forced to become true with a proper $\forceP \in L[G_{\kappa}]$. The choice of $F$ ensures that $\tau $ and its parameters are considered at an even stage $\beta < \kappa$ of our iteration. But $\forceP_{\kappa} \ast \forceP$ is suitable and witnesses that $\tau $ can be forced to become true. Hence we forced so that $\tau$ becomes true already at stage $\beta < \kappa$ and by upwards absoluteness of $\Sigma_1$-statements the proof is done.
\end{proof}

\end{document}